\documentclass{amsart}
\usepackage{graphicx}
\usepackage{amscd}
\usepackage{color}
\newtheorem{thm}{Theorem}

 \newtheorem{claim}{Claim}
 
 \newtheorem{prob}{Problem}
\newtheorem*{thm*}{Theorem}

\begin{document}

\title[Neighborhood equivalence for multibranched surfaces]
{Neighborhood equivalence for multibranched surfaces in 3-manifolds}

\author{Kai Ishihara}
\thanks{The first author is partially supported by JSPS KAKENHI Grant Numbers 26310206, 16H03928, 16K13751, 17K14190 and 17H06463.}
\address{
Faculty of Education \newline
\indent Yamaguchi University, 1677-1 Yoshida, Yamaguchi-shi, Yamaguchi, 753-8511, Japan}
\email{kisihara@yamaguchi-u.ac.jp}

\author{Yuya Koda}
\thanks{The second author is partially supported by JSPS KAKENHI Grant
 Numbers 15H03620, 17K05254, 17H06463, and JST CREST Grant Number JPMJCR17J4.}
\address{
Department of Mathematics \newline
\indent Hiroshima University, 1-3-1 Kagamiyama, Higashi-Hiroshima, 739-8526, Japan}
\email{ykoda@hiroshima-u.ac.jp}

\author{Makoto Ozawa}
\thanks{The third author is partially supported by JSPS KAKENHI Grant Numbers  26400097, 16H03928, 17K05262.}
\address{
Department of Natural Sciences, Faculty of Arts and Sciences \newline
\indent Komazawa University, 1-23-1 Komazawa, Setagaya-ku, Tokyo, 154-8525, Japan}
\email{w3c@komazawa-u.ac.jp}

\author{Koya Shimokawa}
\thanks{The last author is partially supported by JSPS KAKENHI Grant Numbers 16K13751, 16H03928, 17H06460, 17H06463, and JST CREST Grant Number JPMJCR17J4.}
\address{
Department of Mathematics \newline
\indent Saitama University, 255 Shimo-Okubo, Sakura-ku, Saitama-shi, Saitama, 338-8570, Japan}
\email{ kshimoka@rimath.saitama-u.ac.jp}


\keywords{link, punctured sphere, cabling conjecture, incompressible surface, essential surface, multibranched surface}

\begin{abstract}
A multibranched surface is a 2-dimensional polyhedron without vertices. 
We introduce moves for multibranched surfaces embedded in a 3-manifold, which connect 
any two multibranched surfaces with the same regular neighborhoods in finitely many steps. 
\end{abstract}

\maketitle


\section*{Introduction}

In \cite{Suzuki70}, Suzuki defined the notion of {\it neighborhood equivalence} for pairs of
polyhedra $P \subset M$.
Two pairs of polyhedra $P \subset M$ and $P' \subset M'$ are said to be 
neighborhood equivalent
if there exists an orientation preserving homeomorphism of $M$ to $M'$
which takes $N(P)$ to $N(P')$.
It was shown by Makino-Suzuki \cite{MS95} that two spatial graphs (i.e. graphs 
embedded in 3-manifolds) $\Gamma$ and $\Gamma'$ are
neighborhood equivalent if and only if $\Gamma'$ is obtained from $\Gamma$
by a finite sequence of edge-contractions and vertex-expansions.
This shows that an equivalence class of handlebodies embedded in a 3-manifold 
can be identified with that of spatial graphs modulo edge-contractions
and vertex-expansions.
In the present paper, we show an analogous result on 
the neighborhood equivalence for multibranched surfaces. 

A 2-dimensional polyhedron $X$ is said to be a {\it multibranched surface} 
if each point $x$ of $X$ has a regular neighborhood homeomorphic to 
$C_{d_x}
\times [0,1]$, where $C_{d_x}$ is the cone over $d_x$ points. 
The set $B$ of points with $d_x \geq 3$ is a disjoint union of circles, 
called {\it branch loci}, and $B$ separates $X$ into (genuine) surfaces, called {\it regions}. 
In this paper, we always assume that $X$ does not have disk regions. 
This object naturally arises both in the the study of essential surfaces in the exterior of links, e.g. \cite{EO18}, 
and that of tricontinuous or poly-continuous patterns in material science(see the final paragraph of this section).  

Let $X$ be a multibranched surface embedded in a 3-manifold $M$. 
If there exists an annulus region $A$ of $X$ connecting two different branched loci, 
we obtain a new multibranched surface $X'$ with the 
same regular neighborhood as $X$ by shrinking $A$ into the core circle 
(provided $A$ satisfies a certain condition determined by the combinatorial structure of $X$. 
See Section \ref{Sec:Multibranched surfaces and their moves} for the details.)  
We call this operation the {\it IX-move along $A$}. 
The IX-move along a M\"{o}bius band region can be defined as well in a similar way. 
The condition that actually enable us to perform the IX-move along a given 
annulus or M\"{o}bius band region can be described explicitly 
in terms the combinatorial structure of $X$. 
An inverse operation of the IX-move is called an {\it {XI}-move}. 
By the construction, it is clear that any two multibranched surfaces in 
$M$ connected by a sequence of IX- and XI-moves 
have the same regular neighborhood. 
Our main theorem below claims that these moves are already sufficient to connect 
any two multibranched surfaces with the same regular neighborhood.  

\begin{thm*}[Theorem \ref{thm:IXXI}]
Let $X,X'$ be multibranched surfaces in an orientable $3$-manifold $M$, 
and let $N,N'$ be their regular neighborhoods respectively. 
If $N$ is isotopic to $N'$ in $M$, then $X$ is transformed into $X'$ by a finite sequence of IX-moves, XI-moves and isotopies.
\end{thm*}

When $d_x$ in the definition of a multibranched surface $X$ 
is at most 3 for each $x \in X$, $X$ is called a {\it tribranched surface}. 
The set of tribranched surfaces in a given 3-manifold $M$ is ``generic'', that is, it forms an open and dense subset 
in the space of all multibranched surfaces in a suitable sense. 
For a tribranched surface, an IX-move followed by an XI-move is called an {\it IH-move}. 
Clearly, an IH-move transforms a tribranched surface into another tribranched surface. 
We prove the above theorem showing that 
IH-moves are sufficient to connect 
any two tribranched surfaces in $M$ with the same regular neighborhood (Theorem \ref{thm:IH}).  

On finite calculi for certain ``generic'' subspaces of 3-manifolds, the following facts are well-known. 
\begin{enumerate}
\item (Luo \cite{Luo97}, Ishii \cite{Ishii08})  
Two trivalent graphs embedded in a 3-manifold has isotopic 
neighborhoods if and only if they are connected by a sequence of {\it IH-moves} (for spatial trivalent graphs) and isotopy. 
(This is a version of \cite{MS95} for trivalent spatial graphs.) 
\item (Matveev \cite{Matveev88}, Piergallini \cite{Piergallini88}) 
Two simple polyhedra embedded in a 3-manifold has isotopic 
neighborhoods if and only if they are connected by a sequence of {\it $2 \leftrightarrow 3$ moves}, {\it $0 \leftrightarrow 2$ moves} moves and isotopy.  
\end{enumerate}
Our Theorem \ref{thm:IH} can be regarded as a 2-dimensional analogue of (1). 
The set of tribranched surfaces forms a subclass of the set of 
simple polyhedra studied in (2). 
Theorem \ref{thm:IH} can also be regarded as a 
study of moves closed in that subset.

In Section \ref{sec:tricontinuous} we discuss an application to study of tricontinuous or poly-continuous patterns.
Sructures called bicontinuous or tricontinuous (more generally poly-continuous) patterns appear, for example, in block copolymer materials.
Tangled networks (labyrinths) in $\mathbb R^3$ are used in the study of poly-continuous structures in \cite{Hyde2000, Hyde2009}.
In this paper we will focus our attention to poly-continuous patterns defined by triply periodic multibranched surfaces in $\mathbb R^3$.
We will give a relation of such poly-continuous patterns associated to a given tangled network (Theorem \ref{thm:poly}).

\section{Multibranched surfaces and their moves}
\label{Sec:Multibranched surfaces and their moves}

Let $X$ be a multibranched surface with brach loci $B=B_1\cup\cdots\cup B_m$ and regions $S=S_1\cup \cdots\cup S_n$, 
where $S$ is a (possibly disconnected or/and non-orientable) compact surface without disk components such that each component $S_j$ ($j=1,\ldots,n$) has a non-empty boundary.
Each point $x$ in $\partial S$ is identified with a point $f(x)$ in $B$ by a covering map $f:\partial S\to B$,
where $f|_{f^{-1}(B_i)}:f^{-1}(B_i)\to B_i$ is a $d_i$-fold covering ($d_i>2$). 
Note that $f^{-1}(B_i)$ might be disconnected.
We call $d_i$ the {\emph degree} of $B_i$.
We say that $B_i$ is {\em tribranched} or a {\em tribranch locus} if $d_i=3$. 
For each component $C$ of $\partial S$, the {\em wrapping number} of $C$ is $w_{C}$ if $f|_{C}$ is a $w_C$-fold covering for the branch locus $f(C)$.
Suppose $X$ is embedded in an orientable $3$-manifold $M$.
By \cite{MO17}, then for each branch locus $B_i$ of $X$, the wrapping number of all component of $f^{-1}(B_i)$ is a divisor of $d_i$. 
We call the divisor $w_i$ the {\em wrapping number} of $B_i$. 
We say a branch locus $B_i$ is {\em normal} (resp. {\em pure}) if $w_i=1$ (resp. $d_i=w_i$).
Note that if $B_i$ is normal (resp. pure), then 
$f^{-1}(B_i)$ consists of $d_i$ components (resp. one component) of $\partial S$.
For each annulus region $A$ of $X$, exactly one of the following 
(1)--(4) holds: 

\begin{itemize}
\item[(1)] $\partial A$ consists of two normal branch loci, 
\item[(2)] $\partial A$ consists of an normal branch locus and an unnormal branch locus, 
\item[(3)] $\partial A$ consists of two unnormal branch loci, or 
\item[(4)] $\partial A$ consists of one branch locus.
\end{itemize}
In the cases of (1), (2), (3), (4), the annulus region $A$ is called a {\em normal annulus region}, {\em quasi-normal annulus region}, 
{\em unnormal annulus region}, {\em closing annulus region}, respectively.
We say a M\"{o}bius-band region is {\em normal} if the boundary is an normal branch locus.
%
We say that a branch locus is {\em non-spreadable}
if it is normal and tribranched, or pure, otherwise we say that it is {\em spreadable}. 
We say that a region is {\em maximally spread} if each boundary component is non-spreadable.
Let $A$ be an normal or quasi-normal annulus region of $X$, or an normal  M\"{o}bius-band region of $X$.
A {\em $($$2$-dimensional$)$ IX-move} along $A$ is an operation shrinking 
$A$ into the core circle. 
By this move, two branch loci become one spreadable
branch locus if $A$ is an normal or quasi-normal annulus region, 
and one normal branch locus becomes one unnormal and non-pure (spreadable) branch locus if $A$ is an normal  M\"{o}bius-band region. 
A {\em $($$2$-dimensional$)$ XI-move} at a spreadable 
branch locus is a reverse operation of an IX-move.
See Figure \ref{fig:IXXI}. 

\begin{figure}[h]
\begin{center}
\begin{minipage}{400pt}
\includegraphics[width=\textwidth]{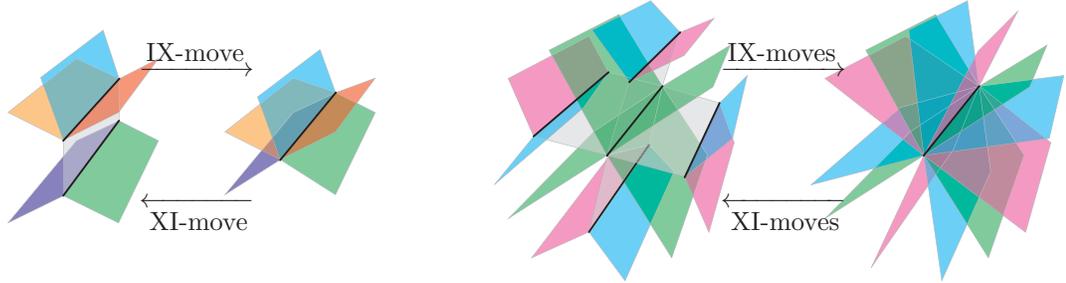}

\begin{picture}(400,0)(0,0)
\put(50,90){$\xrightarrow[]{\mbox{IX-move}}$}
\put(50,40){$\xleftarrow[\mbox{XI-move}]{}$}
\put(270,90){$\xrightarrow[]{\mbox{IX-moves}}$}
\put(270,40){$\xleftarrow[\mbox{XI-moves}]{}$}
\end{picture}
\end{minipage}
\caption{IX-moves and XI-moves: an IX-move along an normal annulus region and an XI-move at an normal branch locus (left), 
an IX-move along a quasi-normal annulus region and an XI-move at an unnormal branch locus (right).}
\label{fig:IXXI}
\end{center}
\end{figure}

By an XI-move, a new normal or quasi-normal annulus region, or a new normal M\"{o}bius band region arises.
An IX-move is uniquely determined up to isotopy for a given normal or 
quasi-normal annulus region, or a given normal  M\"{o}bius band region. 
An XI-move, however, is not uniquely determined for an  spreadable 
branch locus. 
 Note that a branch locus admits an XI-move if and only if 
it is spreadable. 
Suppose the region $A$ above is maximally spread. 
Note that each component (branch locus) of $\partial A$ does not admit XI-moves. 
Perform the IX-move along $A$. 
Then the resulting new branch locus admits exactly two XI-moves. 
One is the reverse operation of the IX-move.  
A {\em $($$2$-dimensional$)$ IH-move} along $A$ is a composition of the IX-move along $A$ and the other XI-move at the new branch locus,
see Figure \ref{fig:IH}. 

\begin{figure}[h]
\begin{center}
\begin{minipage}{400pt}
\includegraphics[width=\textwidth]{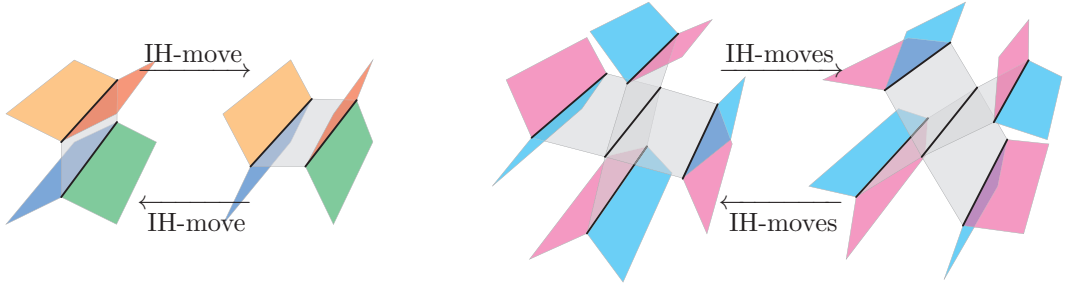}
\begin{picture}(400,0)(0,0)
\put(50,90){$\xrightarrow[]{\mbox{IH-move}}$}
\put(50,40){$\xleftarrow[\mbox{IH-move}]{}$}
\put(270,90){$\xrightarrow[]{\mbox{IH-moves}}$}
\put(270,40){$\xleftarrow[\mbox{IH-moves}]{}$}
\end{picture}
\end{minipage}
\caption{IH-moves: IH-moves along normal maximally spread 
annulus (or M\"{o}bius band) regions (left), 
IH-moves along quasi-normal maximally spread 
annulus regions (right).}
\label{fig:IH}
\end{center}
\end{figure}

These moves invariant regular neighborhoods of the mutibranched surfaces up to isotopy.
The following is our main theorem, which implies that the converse is also true.
\begin{thm}\label{thm:IXXI}
Let $X,X'$ be multibranched surfaces in an orientable $3$-manifold $M$, 
and let $N,N'$ be their regular neighborhoods respectively. 
If $N$ is isotopic to $N'$ in $M$, then $X$ is transformed into $X'$ by a finite sequence of IX-moves, XI-moves and isotopies.
\end{thm}

\section{Proof of Theorem \ref{thm:IXXI}}
\label{Sec:Proof of Theorem {thm:IXXI}}

To prove Theorem \ref{thm:IXXI} we may assume that $N=N'$, namely $X'$ is in $N$ so that $N$ is also a regular neighborhood of $X'$. 
The {\em characteristic annulus system $\mathcal{A}_X$ 
$($resp. $\mathcal{A}_{X'})$ with respect to $X$} (resp. $X'$) 
is the system of mutually disjoint annuli properly embedded in $N$ 
consisting of $\widehat{C} \times [-1,1]$ for each component $C$ of $\partial S$, where 
$\widehat{C}$ is in Int($S$) and parallel to $C$ in $S$. 
The system $\mathcal{A}_X$ splits $N$ into pieces $N_i$ ($i=1,\ldots,m+n$), 
where $N_i$ is a regular neighborhoods of a branch locus $B_i$ of $X$ for $i\in \{1,\ldots,m\}$, and $N_{m+j}$ is homeomorphic to 
$S_j\times [-1,1]$ or $S_j\tilde{\times} [-1,1]$ according to whether $S_j$ is orientable or not for $j\in\{1,\ldots,n\}$. 
In other words, $N$ is obtained by attaching $\bigcup_{1\le j\le n}N_{m+j}$ 
to the union $\bigcup_{1\le i\le m}N_{i}$ of the solid tori  along $\partial S\times [-1,1]$.

An IX-move along a region $A$ of a multibranched surface $X$ in $N$ 
corresponds to an operation removing one or two annuli $\partial A\times [-1,1]$ from 
$\mathcal{A}_X$ according to whether $A$ is a (normal) M\"{o}bius-band region or a (normal/quasi-normal) annulus region.
On the other hand, an IX-move in $N$ corresponds to an operation adding 
two parallel annuli in a 
 solid torus $N_i$ ($i\in\{1,\ldots,m\}$) 
to $\mathcal{A}_X$. 
Note that each spreadable
branch locus of $X$ admits an XI-move.
By applying XI-moves to $X$ (resp. $X'$) maximally, we get a multibranched surface with the non-spreadable 
branch loci.
We call such a multibranched surface a 
{\em maximally spread surface}.
Theorem \ref{thm:IXXI} then follows from Theorem \ref{thm:IH} below. 
\begin{thm}\label{thm:IH}
Let $X$ and $X'$ be maximally spread surfaces
in $N$ such that $N$ is a regular neighborhood of each of $X$ and $X'$.
Then $X$ is transformed into $X'$ by a finite sequence of IH-moves and isotopies.
\end{thm}
Let $\mathcal{A}_X$ and $\mathcal{A}_{X'}$ be the characteristic annulus system with respect to $X$ and $X'$, respectively. 
We assume that the annuli of $\mathcal{A}_X$ and $\mathcal{A}_{X'}$ intersect transversely and minimally up to isotopy.  
If $( \bigcup_{A \in \mathcal{A}_X} A ) \cap ( \bigcup_{A' \in \mathcal{A}_{X'}} A' ) 
= \emptyset$, then  
 $X$ and $X'$ are isotopic in $N$ since $\mathcal{A}_X$ coincides with $\mathcal{A}_{X'}$ 
up to isotopy. 
Hence we suppose that $( \bigcup_{A \in \mathcal{A}_X} A ) \cap ( \bigcup_{A' \in \mathcal{A}_{X'}} A' ) \neq \emptyset$. 
\begin{claim}
\label{claim:intersection consists of essential loops}
Any component $\alpha$ of the intersection between $A\in\mathcal{A}_X$ and $A'\in\mathcal{A}_{X'}$
is an loop and essential in each of $A$ and $A'$.
\end{claim}
\begin{proof}
First, suppose that $\alpha$ is a loop and inessential in $A$ or $A'$, 
say in $A'$. 
We may assume that $\alpha$ is innermost in $A'$. 
Then the disk $D'$ in $A'$ bounded by $\alpha$ is in 
some piece $N_i$ ($i\in\{1,\ldots,m+n\}$). 
Since the core of $A$ is not null-homologous in $N_i$, 
$\alpha$ is also inessential in $A$. 
Thus $\alpha$ can be removed by isotopy in $N$, which is a contradiction.
Next, suppose that $\alpha$ is an arc and inessential in $A$ or $A'$, say in $A'$.
We may assume that $\alpha$ is an outermost arc in $A'$. 
The disk $D'$ cut off from $A'$ by $\alpha$ is in 
some piece $N_i$ ($i\in\{1,\ldots,m+n\}$).
Then $\alpha$ is also inessential in $A$, and so can be removed by isotopy in $N$, 
a contradiction.
Finally, suppose that $\alpha$ is an arc and essential in each of $A$ and $A'$. 
Let $D'$ be the component of $A'$ cut off by $A\in\mathcal{A}$ 
such that $\alpha \subset \partial D$ and $D'$ lies in some solid torus $N_i$ ($i\in\{1,\ldots,m\}$). 
$\partial D'$ intersects $A\in\mathcal{A}$ in two essential arcs in $A'$. 
This implies the degree $d_i$ of $B_i$ is $2$, that is a contradiction.
\end{proof}


By Claim \ref{claim:intersection consists of essential loops}, each component of 
$( \bigcup_{A \in \mathcal{A}_X} A ) \cap ( \bigcup_{A' \in \mathcal{A}_{X'}} A' )$ 
is a loop essential both in $\bigcup_{A \in \mathcal{A}_X} A$ and 
$\bigcup_{A' \in \mathcal{A}_{X'}} A'$.
For such systems $\mathcal{A},\mathcal{A'}$ of annuli in $N$, 
we denote by $I(\mathcal{A},\mathcal{A'})$ the number of components of the intersection 
$( \bigcup_{A \in \mathcal{A}} A ) \cap ( \bigcup_{A' \in \mathcal{A'}} A' )$, 
and by $End(\mathcal{A'}|\mathcal{A})$ the set of outermost annuli cut off from annuli in $\mathcal{A'}$ along loop intersections. 
Put $E(\mathcal{A'}|\mathcal{A}):=|End(\mathcal{A'}|\mathcal{A})|$. 
We will prove Theorem \ref{thm:IH} by induction on the complexity $(E(\mathcal{A}_{X'}|\mathcal{A}_{X}),I(\mathcal{A}_{X'},\mathcal{A}_{X}))$ in the lexicographic order.
Each annulus of $\mathcal{A}_{X'}$ having nonempty intersection with annuli of $\mathcal{A}_X$ contains exact two elements of $End(\mathcal{A_{A'}}|\mathcal{A}_{X})$, so $E(\mathcal{A}_{X'}|\mathcal{A}_{X})$ is twice the number of such annuli of $\mathcal{A}_{X'}$.
We may assume that any annuli of $\mathcal{A}_{X'}$ without intersection have been moved into $\bigcup_{j}N_{m+j}$.
Take an element $E$ of $End(\mathcal{A}_{X'}|\mathcal{A}_X)$, say,  $E\subset A'_{1}\in \mathcal{A}_{X'}$, $E\cap ( \bigcup_{A \in \mathcal{A}_{X}} A )=\partial E\cap A_1=\alpha$, and $\beta = \partial E - \alpha$.  
For $i\in\{1,\ldots,m\}$, we say that the solid torus $N_i$ is {\em normal} or {\em pure} 
if $B_i$ is so. 

\begin{claim}\label{claim:end}
$E$ is in an normal solid torus, say $N_1$. 
\end{claim}
\begin{proof}
Suppose that $E$ is in $N_{m+j}$ ($j\in\{1,\ldots,n\}$).
Recall that $N_{m+j}$ is homeomorphic to $S_j\times[-1,1]$ or $S_j\tilde{\times}[-1,1]$.
Since $\alpha$ is in $\partial S_j\times [-1,1]$ or $S_j\tilde{\times}[-1,1]$ and 
$\beta$ is in $S_j\times\{-1,1\}$ or $S_j\tilde{\times}\{-1,1\}$, 
the intersection $\alpha$ can be removed by isotopy in $N$.
This is a contradiction.
Thus $E$ is in some solid torus, say $N_1$. 
The union $\alpha\cup\beta$ is a $(2d_1,2e_1)$-torus link in $\partial N_1$, where $d_1$ and $e_1$ are relative prime integers.
The annulus $E$ bounded by $\alpha\cup\beta$ is boundary-parallel in $N_1$. 
If $\partial N_1$ contains no annulus of $\mathcal{A}_X$ other than the annulus containing $\alpha$, 
the intersection $\alpha$ can be removed by isotopy in $N$.
This is again a contradiction.
Hence $\partial N_1$ contains more than two annuli of $\mathcal{A}_X$.
This implies that $N_1$ is not pure, thus it is normal.
\end{proof}

%
 
Let $A_2,A_3$ be the annuli of $\mathcal{A}_X$ lying in $\partial N_1$ other than $A_1$. 
See Figure \ref{fig:N_1}. 
\begin{figure}[h]
\begin{center}
\begin{minipage}{3cm}
\includegraphics[width=\textwidth]{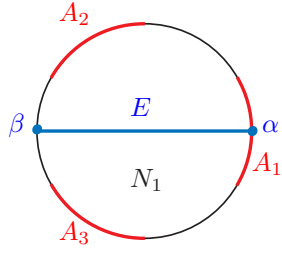}
\begin{picture}(400,0)(0,0)
\put(83,38){\color{red} $A_1$}
\put(10,95){\color{red} $A_2$}
\put(10,12){\color{red} $A_3$}
\put(37,59){\color{blue} $E$}
\put(87,53){\color{blue} $\alpha$}
\put(-9,53){\color{blue} $\beta$}
\put(37,32){$N_1$}
\end{picture}
\end{minipage}
\caption{The normal solid torus $N_1$.}
\label{fig:N_1}
\end{center}
\end{figure}
\begin{claim}\label{claim:intersectN_1}
$I(\mathcal{A}_{X'},\{A_1\})>I(\mathcal{A}_{X'},\{A_2\})+I(\mathcal{A}_{X'},\{A_3\})$.
\end{claim}
\begin{proof}
Let $\tilde{A}_1,\tilde{A}_2, \tilde{A}_3$ be components of 
the closure of $\partial N_1-(A_1\cup A_2\cup A_3)$ 
such that $\tilde{A}_i$ and $A_i$ are disjoint for each $i\in\{1,2,3\}$.
Let $b_{jk}$ (resp. $b_{jj}$) be the number of (annulus) components of $( \bigcup_{A' \in \mathcal{A}_{X'}} A' )\cap N_1$ 
whose boundary lies in $A_j\cup A_k$ for $j\neq k$ 
(resp. $A_j\cup \tilde{A}_j$ for each $j\in\{1,2,3\}$). 
Then  $I(\mathcal{A}_{X'},\{A_j\})=\sum_{k} b_{jk}$ for each $j\in\{1,2,3\}$.
The annulus $E$ with boundaries $\alpha\in A_1$,  $\beta\in\tilde{A}_1$ is counted in $b_{11}$, and so $b_{11} > 0$.
Any annulus component of $( \bigcup_{A' \in \mathcal{A}_{X'}} A' )\cap N_1$ other than $E$ 
has to be contained in one of the two components of 
$N_1- E$, otherwise intersects $E$. 
Thus we have $b_{22} = b_{33} = b_{23} = 0$. 
Further, we have  
$I(\mathcal{A}_{X'},\{A_1\})= b_{11} + b_{12} + b_{13} =b_{11}+I(\mathcal{A}_{X'},\{A_2\})+I(\mathcal{A}_{X'},\{A_3\})>I(\mathcal{A}_{X'},\{A_2\})+I(\mathcal{A}_{X'},\{A_3\})$. 
\end{proof}
We may assume that $A_1 \subset N_1\cap N_{m+1}$.
Recall that the piece $N_{m+1}$ corresponds to the region $S_1$, 
which is an normal M\"{o}bius band region, normal annulus region, quasi-normal annulus region or closing annulus region.  
Here we know $S_1$ is neither an unnormal M\"{o}bius band region nor
 an unnormal annulus region since $B_1$ is normal by Claim \ref{claim:end}.
\begin{claim}\label{claim:notorus}
$S_1$ is not a closing annulus 
region.
\end{claim}
\begin{proof}
Suppose that $S_1$ is a closing annulus 
region. 
It means that $N_{1}\cap N_{m+1}=A_1\cup A_2$ or $A_1\cup A_3$, say, $A_1\cup A_2$. 
Two boundary components of each annulus component of $( \bigcup_{A' \in \mathcal{A}_{X'}} A' )\cap N_{m+1}$ must lie $A_1$ and $A_2$, respectively. Then $I(\mathcal{A}_{X'},\{A_1\})=I(\mathcal{A}_{X'},\{A_2\})$, that contradicts Claim \ref{claim:intersectN_1}.
\end{proof}

If $S_1$ is an normal M\"{o}bius band region, let $N_{12}$ be the solid torus $N_1\cup N_{m+1}$. 
If $S_1$ is an annulus region, by Claim \ref{claim:notorus}, $\partial S_1$ has one branch locus other than $B_1$, say, $B_2$, so let $N_{12}$ be the solid torus $N_1\cup N_{m+1}\cup N_2$, and put  $A_4:=N_2\cap N_{m+1}$, which is an annulus of $\mathcal{A}_X$.
Consider the maximally spread 
surface $X^{(1)}$ obtained from $X$ by an IH-move along $S_1$.
If $S_1$ is an normal M\"{o}bius band region, the characteristic annulus system $\mathcal{A}_{X^{(1)}}$ with respect to $X^{(1)}$
is obtained from $\mathcal{A}_{X}$ by replacing $A_1$ with an annulus 
$A_1^{(1)}$ 
in $N_{12}$ disjoint from any annuli of $\mathcal{A}_{X}-\{A_1\}$. 
See Figure \ref{fig:IH-move_mobius_band}. 
\begin{figure}[h]
\begin{center}
\begin{minipage}{10.5cm}
\includegraphics[width=\textwidth]{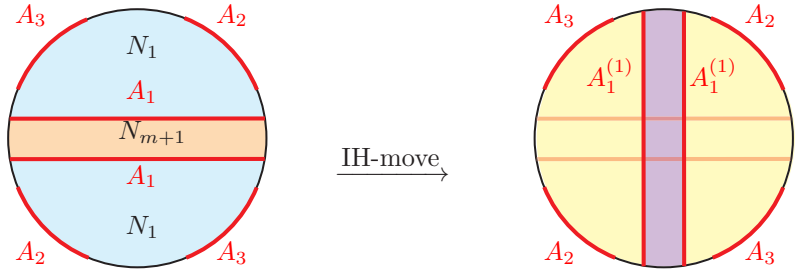}
\begin{picture}(400,0)(0,0)
\put(124,45){$\xrightarrow[]{\mbox{IH-move}}$}
\put(45,44){\color{red} $A_1$}
\put(45,76){\color{red} $A_1$}
\put(79,105){\color{red} $A_2$}
\put(3,15){\color{red} $A_2$}
\put(3,105){\color{red} $A_3$}
\put(79,15){\color{red} $A_3$}
\put(45,93){$N_1$}
\put(45,25){$N_1$}
\put(42,61){$N_{m+1}$}
\put(219,81){\color{red} $A_1^{(1)}$}
\put(258,81){\color{red} $A_1^{(1)}$}
\put(279,105){\color{red} $A_2$}
\put(203,15){\color{red} $A_2$}
\put(203,105){\color{red} $A_3$}
\put(279,15){\color{red} $A_3$}
\end{picture}
\end{minipage}
\caption{The IH-move along $S_1$ when $S_1$ is an normal M\"{o}bius-band region.}
\label{fig:IH-move_mobius_band}
\end{center}
\end{figure}
If, on the other hand, $S_1$ is an normal or quasi-normal annulus region, 
the characteristic annulus system $\mathcal{A}_{X^{(1)}}$ with respect to $X^{(1)}$
is obtained from $\mathcal{A}_{X}$ by replacing $\{A_1,A_4\}$ with 
parallel annuli $\{A_1^{(1)},A_4^{(1)}\}$ 
in $N_{12}$ disjoint from any annuli of $\mathcal{A}_{X}-\{A_1,A_4\}$. 
See Figure \ref{fig:IH-move_mobius_band}. 
\begin{figure}[h]
\begin{center}
\begin{minipage}{400pt}
\includegraphics[width=\textwidth]{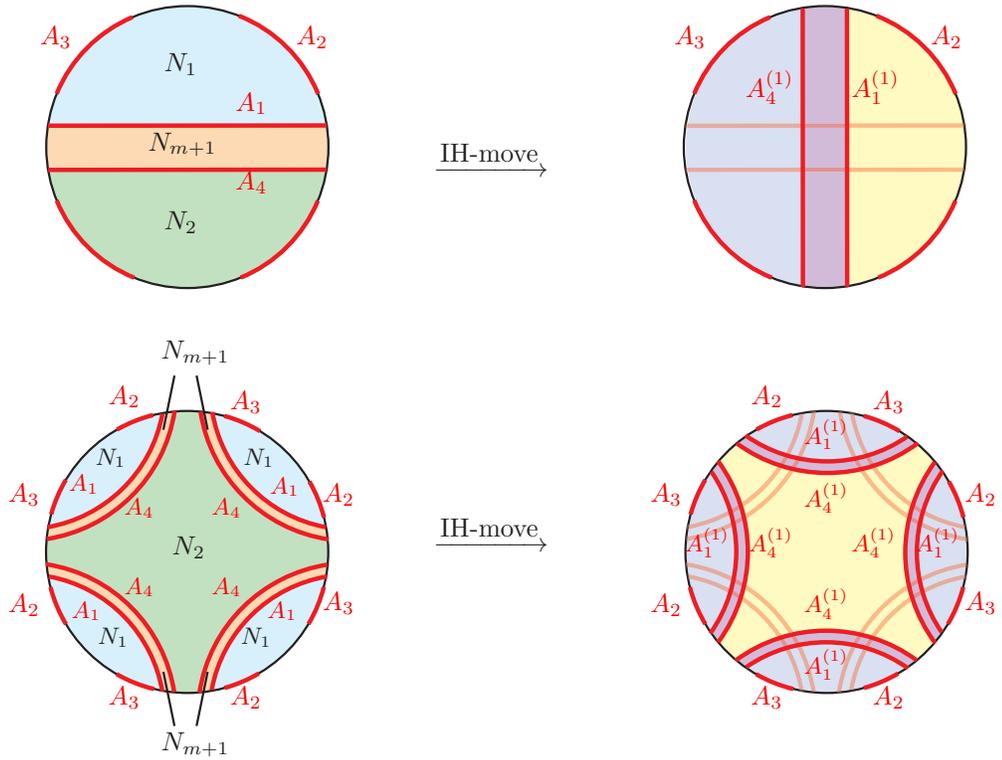}
\begin{picture}(400,0)(0,0)
\put(172,220){$\xrightarrow[]{\mbox{IH-move}}$}
\put(172,78){$\xrightarrow[]{\mbox{IH-move}}$}
\put(97,245){\color{red} $A_1$}
\put(120,270){\color{red} $A_2$}
\put(23,270){\color{red} $A_3$}
\put(97,215){\color{red} $A_4$}
\put(330,250){\color{red} $A_1^{(1)}$}
\put(360,270){\color{red} $A_2$}
\put(263,270){\color{red} $A_3$}
\put(290,250){\color{red} $A_4^{(1)}$}
\put(70,260){$N_1$}
\put(70,200){$N_2$}
\put(64,230){$N_{m+1}$}
\put(34,101){\small \color{red} $A_1$}
\put(110,100){\small \color{red} $A_1$}
\put(35,53){\small \color{red} $A_1$}
\put(109,53){\small \color{red} $A_1$}
\put(130,97){\color{red} $A_2$}
\put(49,134){\color{red} $A_2$}
\put(11,55){\color{red} $A_2$}
\put(95,20){\color{red} $A_2$}
\put(95,132){\color{red} $A_3$}
\put(11,97){\color{red} $A_3$}
\put(49,20){\color{red} $A_3$}
\put(130,56){\color{red} $A_3$}
\put(88,92){\small \color{red} $A_4$}
\put(55,92){\small \color{red} $A_4$}
\put(88,62){\small \color{red} $A_4$}
\put(55,62){\small \color{red} $A_4$}
\put(267,78){\small \color{red} $A_1^{(1)}$}
\put(312,119){\small \color{red} $A_1^{(1)}$}
\put(354,78){\small \color{red} $A_1^{(1)}$}
\put(312,32){\small \color{red} $A_1^{(1)}$}
\put(373,97){\color{red} $A_2$}
\put(292,134){\color{red} $A_2$}
\put(254,55){\color{red} $A_2$}
\put(338,20){\color{red} $A_2$}
\put(338,132){\color{red} $A_3$}
\put(254,97){\color{red} $A_3$}
\put(292,20){\color{red} $A_3$}
\put(373,56){\color{red} $A_3$}
\put(291,78){\small \color{red} $A_4^{(1)}$}
\put(312,95){\small \color{red} $A_4^{(1)}$}
\put(330,78){\small \color{red} $A_4^{(1)}$}
\put(312,55){\small \color{red} $A_4^{(1)}$}
\put(69,3){$N_{m+1}$}
\put(69,151){$N_{m+1}$}
\put(73,77){$N_{2}$}
\put(44,111){\small $N_1$}
\put(100,111){\small $N_1$}
\put(45,43){\small $N_1$}
\put(99,43){\small $N_1$}
\end{picture}
\end{minipage}
\caption{The IH-move along $S_1$ when $S_1$ is an normal or quasi-normal annulus region.}
\label{fig:IH-move_annulus}
\end{center}
\end{figure}
 

Let $E^*$ be the component of $A'_{1}\cap N_{12}$ which contains the annulus $E$ and $\gamma=\partial E^*-\beta$.
Then exactly one of the following holds: 
\begin{description}
\item{Case $1$:}  $E^*=A'_{1}$, see Figure \ref{fig:Case_1}.
\begin{figure}[h]
\begin{center}
\begin{minipage}{12cm}
\includegraphics[width=\textwidth]{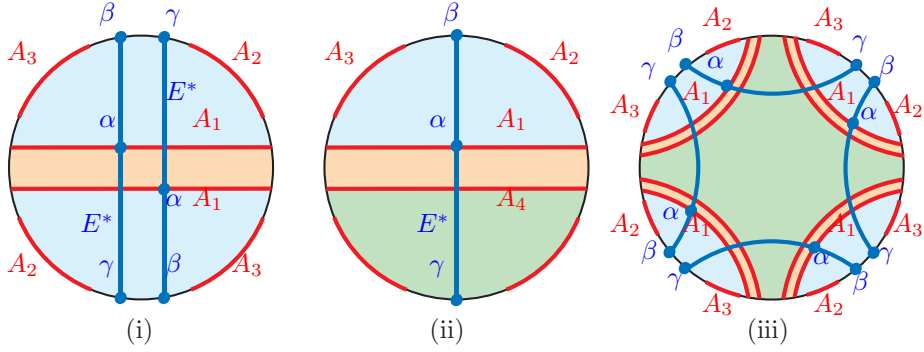}
\begin{picture}(400,0)(0,0)
\put(70,80){\color{red} $A_1$}
\put(70,50){\color{red} $A_1$}
\put(0,105){\color{red} $A_3$}
\put(85,105){\color{red} $A_2$}
\put(0,25){\color{red} $A_2$}
\put(85,25){\color{red} $A_3$}
\put(28,40){\color{blue} $E^*$}
\put(35,80){\color{blue} $\alpha$}
\put(35,120){\color{blue} $\beta$}
\put(35,25){\color{blue} $\gamma$}
\put(60,90){\color{blue} $E^*$}
\put(60,50){\color{blue} $\alpha$}
\put(60,120){\color{blue} $\gamma$}
\put(60,25){\color{blue} $\beta$}
\put(185,80){\color{red} $A_1$}
\put(185,50){\color{red} $A_4$}
\put(120,105){\color{red} $A_3$}
\put(205,105){\color{red} $A_2$}
\put(155,40){\color{blue} $E^*$}
\put(160,80){\color{blue} $\alpha$}
\put(165,120){\color{blue} $\beta$}
\put(160,25){\color{blue} $\gamma$}
\put(228,85){\color{red} $A_3$}
\put(263,118){\color{red} $A_2$}
\put(255,90){\color{red} $A_1$}
\put(305,118){\color{red} $A_3$}
\put(335,85){\color{red} $A_2$}
\put(310,90){\color{red} $A_1$}
\put(228,40){\color{red} $A_2$}
\put(263,10){\color{red} $A_3$}
\put(255,40){\color{red} $A_1$}
\put(305,10){\color{red} $A_2$}
\put(335,40){\color{red} $A_3$}
\put(310,40){\color{red} $A_1$}
\put(265,102){\color{blue} $\alpha$}
\put(250,110){\color{blue} $\beta$}
\put(240,100){\color{blue} $\gamma$}
\put(323,83){\color{blue} $\alpha$}
\put(330,100){\color{blue} $\beta$}
\put(320,110){\color{blue} $\gamma$}
\put(248,45){\color{blue} $\alpha$}
\put(240,30){\color{blue} $\beta$}
\put(250,20){\color{blue} $\gamma$}
\put(305,28){\color{blue} $\alpha$}
\put(320,18){\color{blue} $\beta$}
\put(330,30){\color{blue} $\gamma$}
\put(45,0){(i)}
\put(160,0){(ii)}
\put(280,0){(iii)}
\end{picture}
\end{minipage}
\caption{The annulus $E^*$ in Case 1 when $S_1$ is (i) an normal M\"{o}bius-band region; (ii) an normal annulus region; 
(iii) a quasi-normal annulus region. }
\label{fig:Case_1}
\end{center}
\end{figure}
\item{Case $2$:}  $E^*\neq A'_{1}$ and $S_1$ is an normal M\"{o}bius-band region or a quasi-normal annulus region, see Figure \ref{fig:Case_2}.
\begin{figure}[h]
\begin{center}
\begin{minipage}{12cm}
\includegraphics[width=\textwidth]{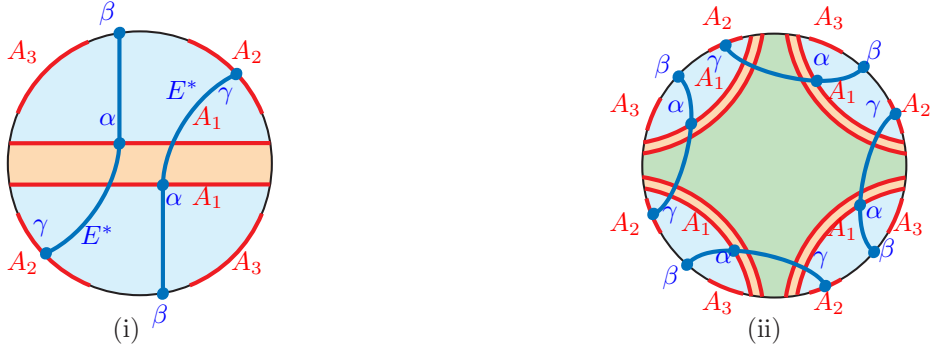}
\begin{picture}(400,0)(0,0)
\put(70,80){\color{red} $A_1$}
\put(70,50){\color{red} $A_1$}
\put(0,105){\color{red} $A_3$}
\put(85,105){\color{red} $A_2$}
\put(0,25){\color{red} $A_2$}
\put(85,25){\color{red} $A_3$}
\put(28,35){\color{blue} $E^*$}
\put(35,80){\color{blue} $\alpha$}
\put(35,120){\color{blue} $\beta$}
\put(10,40){\color{blue} $\gamma$}
\put(60,90){\color{blue} $E^*$}
\put(60,50){\color{blue} $\alpha$}
\put(80,90){\color{blue} $\gamma$}
\put(55,5){\color{blue} $\beta$}
\put(228,85){\color{red} $A_3$}
\put(263,118){\color{red} $A_2$}
\put(260,95){\color{red} $A_1$}
\put(305,118){\color{red} $A_3$}
\put(338,85){\color{red} $A_2$}
\put(310,90){\color{red} $A_1$}
\put(228,40){\color{red} $A_2$}
\put(263,10){\color{red} $A_3$}
\put(255,40){\color{red} $A_1$}
\put(305,10){\color{red} $A_2$}
\put(335,40){\color{red} $A_3$}
\put(310,38){\color{red} $A_1$}
\put(265,102){\color{blue} $\gamma$}
\put(250,85){\color{blue} $\alpha$}
\put(245,100){\color{blue} $\beta$}
\put(325,85){\color{blue} $\gamma$}
\put(305,103){\color{blue} $\alpha$}
\put(325,105){\color{blue} $\beta$}
\put(248,45){\color{blue} $\gamma$}
\put(268,28){\color{blue} $\alpha$}
\put(248,20){\color{blue} $\beta$}
\put(305,28){\color{blue} $\gamma$}
\put(325,45){\color{blue} $\alpha$}
\put(330,30){\color{blue} $\beta$}
\put(40,0){(i)}
\put(280,0){(ii)}
\end{picture}
\end{minipage}
\caption{The annulus $E^*$ in Case 2 when $S_1$ is (i) an normal M\"{o}bius-band region; (ii) an annulus region. }
\label{fig:Case_2}
\end{center}
\end{figure}
\item{Case $3$:}  $E^*\neq A'_{1}$ and  $S_1$ is an normal annulus region, see Figure \ref{fig:Case_3}.
\begin{figure}[h]
\begin{center}
\begin{minipage}{6cm}
\includegraphics[width=\textwidth]{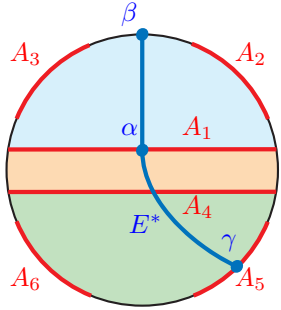}

\begin{picture}(400,0)(0,0)
\put(100,77){\color{red} $A_1$}
\put(100,47){\color{red} $A_4$}
\put(35,105){\color{red} $A_3$}
\put(35,20){\color{red} $A_6$}
\put(120,20){\color{red} $A_5$}
\put(120,105){\color{red} $A_2$}
\put(80,40){\color{blue} $E^*$}
\put(77,77){\color{blue} $\alpha$}
\put(77,120){\color{blue} $\beta$}
\put(115,35){\color{blue} $\gamma$}
\end{picture}
\end{minipage}
\caption{The annulus $E^*$ in Case 3.}
\label{fig:Case_3}
\end{center}
\end{figure}
\end{description}

\begin{claim}\label{claim:case1}
In Case $1$, we have $E(\mathcal{A}_{X'}|\mathcal{A}_{X^{(1)}})<E(\mathcal{A}_{X'}|\mathcal{A}_{X})$.
\end{claim}
\begin{proof}
The annuli $A_1^{(1)}$ (and $A_4^{(1)}$ as well if $S_1$ is an annulus region) are isotopic to $E^*=A'_1$ in $N$.
Then  $End(\mathcal{A}_{X'}|\mathcal{A}_{X^{(1)}})\subset End(\mathcal{A}_{X'}|\mathcal{A}_{X})$ 
and $E\in End(\mathcal{A}_{X'}|\mathcal{A}_{X})-End(\mathcal{A}_{X'}|\mathcal{A}_{X^{(1)}})$.
Hence $E(\mathcal{A}_{X'}|\mathcal{A}_{X^{(1)}})<E(\mathcal{A}_{X'}|\mathcal{A}_{X})$.
\end{proof}

For Case $2$, since $A_2$ and $A_3$ are the only annuli of $\mathcal{A}_X$ lying in $\partial N_{12}$, 
the loop $\gamma$ lies in $A_2$ or $A_3$, say, $A_2$.
\begin{claim}\label{claim:case2}
In Case $2$, we have 
$E(\mathcal{A}_{X'}|\mathcal{A}_{X^{(1)}})\le E(\mathcal{A}_{X'}|\mathcal{A}_{X})$,  
$I(\mathcal{A}_{X'},\{A_1^{(1)}\})< I(\mathcal{A}_{X'},\{A_2\})$ 
and $I(\mathcal{A}_{X'},\mathcal{A}_{X^{(1)}})< I(\mathcal{A}_{X'},\mathcal{A}_{X})$ in Case $2$.
\end{claim}
\begin{proof}
The annulus $A_2$ is divided into two annuli by $\gamma$. 
One of the annuli, denoted by $F$, makes the annulus $E^*\cup F$ isotopic to $A_{1}^{(1)}$ 
 in $N$.
Then each annulus of $\mathcal{A}_{X'}$ that intersects annuli of $\mathcal{A}_{X^{(1)}}$ 
must intersect annuli of $\mathcal{A}_{X}$. 
This implies that  $E(\mathcal{A}_{X'}|\mathcal{A}_{X^{(1)}})\le E(\mathcal{A}_{X'}|\mathcal{A}_{X})$. 
Since $F$ is contained in $A_2$, $I(\mathcal{A}_{X'},\{A_1^{(1)}\})< I(\mathcal{A}_{X'},\{A_2\})$.
By Claim \ref{claim:intersectN_1}, we have $
I(\mathcal{A}_{X'},\{A_1^{(1)}\})<I(\mathcal{A}_{X'},\{A_1\})
$. 
If $S_1$ is an annulus region,  $I(\mathcal{A}_{X'},\{A_4^{(1)}\})<I(\mathcal{A}_{X'},\{A_4\})$ 
since the two annuli $A_1$ and $A_4$ (resp. $A_1^{(1)}$ and $A_4^{(1)}$) are parallel. 
Hence we have $I(\mathcal{A}_{X'},\mathcal{A}_{X^{(1)}})< I(\mathcal{A}_{X'},\mathcal{A}_{X})$.
\end{proof}

For Case $3$, let $A_5,A_6$ be the annuli of $\mathcal{A}_X$ lying in $\partial N_2$ other than $A_4$ so that $\gamma\subset A_5$.
By the same argument in the proof of Claim \ref{claim:case2}, we have the following.
\begin{claim}\label{claim:case3}
In Case $3$, we have $E(\mathcal{A}_{X'}|\mathcal{A}_{X^{(1)}})\le E(\mathcal{A}_{X'}|\mathcal{A}_{X})$ and
$I(\mathcal{A}_{X'},\{A_1^{(1)}\})=I(\mathcal{A}_{X'},\{A_4^{(1)}\})< I(\mathcal{A}_{X'},\{A_5\})$.
\end{claim}
%

\begin{claim}\label{claim:reducing complexity}
The complexity $I(\mathcal{A}_{X'},\mathcal{A}_{X})$ decreases after a finite sequence of IH-moves for $\mathcal{A}_{X'}$. 
\end{claim}
\begin{proof}
By Claims \ref{claim:case1} and \ref{claim:case2}, the complexity $I(\mathcal{A}_{X'},\mathcal{A}_{X})$ decreases 
after performing the IH-move along $S_1$, so we are done. 
The complexity $I(\mathcal{A}_{X'},\mathcal{A}_{X})$ may increase after the IH-move along $S_1$ 
in Case $3$.
By the same argument as in Luo \cite{Luo97}, however, we conclude that the complexity $(E(\mathcal{A}_{X'}|\mathcal{A}_{X}),I(\mathcal{A}_{X'},\mathcal{A}_{X}))$ is reduced after a finite number of IH-moves as follows. 
The annulus $E^*$ in Case $3$ as well as in Case $2$ is an element of 
$End(\mathcal{A}_{X'}|\mathcal{A}_{X^{(1)}})$. 
We repeat the same argument using the annulus $E^*$ among annuli of $End(\mathcal{A}_{X'}|\mathcal{A}_{X^{(1)}})$.  
If Case $1$ occurs for $E^*$, the complexity decreases, so we are done. 
Therefore, it remains to show that after repeating the same process (IH-moves) finitely many times, we finally 
get an outer most annulus of Case $1$. 
To prove this, let us exam the change in the $n$-tuple $(a_1,\dots,a_n)$ of non-negative integers 
\begin{eqnarray*}
a_j= \left\{ 
\begin{array}{ll}
\frac{1}{2}I(\mathcal{A}_{X'},\mathcal{A}_{S_j}) & \mbox{if $S_j$ is an annulus region},\\
I(\mathcal{A}_{X'},\mathcal{A}_{S_j}) & \mbox{otherwise},
\end{array}
\right.
\end{eqnarray*}
where $\mathcal{A}_{S_j}$ is the set of all annuli lying in $\partial N_{m+j}$.
By Claims \ref{claim:case2} and \ref{claim:case3}, at the first step (i.e. the IH-move along $S_1$), 
we replace one coordinate of the $n$-tuple $(a_1,\dots,a_n)$, say $a_{i_0}$, by $a_{i_0}^{(1)}$, where $0\le a_{i_0}^{(1)}\le a_{i_1}-1$ for some $i_1\neq i_0$.
Let the new $n$-tuple thus obtained be $(a_1^{(1)},\dots,a_n^{(1)})$.
Now we replace $a_{i_1}^{(1)}$ by $a_{i_1}^{(2)}$, where $0\le a_{i_1}^{(2)}\le a_{i_2}^{(1)}-1$ for some $i_2\neq i_1$.
Suppose in the $k$-th step we obtain the $n$-tuple $(a_1^{(k)},\dots,a_n^{(k)})$, where $0\le a_{i_{k-1}}^{(k)}\le a_{i_k}^{(k-1)}-1$.
Noting that $a_{i_k}^{(k-1)} = a_{i_k}^{(k)}$, we have 
\begin{eqnarray*}
\sum_{i=1}^{n} a_i^{(k)}&=&\sum_{i=1}^{n} a_i^{(k-1)}-a_{i_{k-1}}^{(k-1)}+a_{i_k}^{(k)}\\
&\le&\sum_{i=1}^{n} a_i^{(k-1)}-a_{i_{k-1}}^{(k-1)}+a_{i_{k}}^{(k-1)}-1\\
&=&\sum_{i=1}^{n} a_i^{(k-1)}-a_{i_{k-1}}^{(k-1)}+a_{i_{k}}^{(k)}-1.
\end{eqnarray*}
By the construction, it holds $0\le a_i^{(k)}\le \max\{a_1,\ldots,a_n\}-1$ for all $k$ and $i$. 
Thus from the above inequality, it follows that 
\begin{eqnarray*}
\sum_{i=1}^{n} a_i^{(k)}
&\le& \sum_{i=1}^{n} a_i^{(k-1)}-a_{i_{k-1}}^{(k-1)}+a_{i_{k}}^{(k)}-1 \\
&\le& \sum_{i=1}^{n} a_i^{(k-2)}-a_{i_{k-2}}^{(k-2)}+a_{i_{k}}^{(k)}-2 \\
&\le& \cdots \\
&\le& \sum_{i=1}^{n} a_i-a_{i_0}+a_{i_{k}}^{(k)}-k \\
&\le& \sum_{i=1}^{n} a_i + \max\{a_1,\ldots,a_n\}-k-1.
\end{eqnarray*}

Hence after at most $(\sum_{i=1}^{n}a_i+\max\{a_1,\ldots,a_n\})$ steps, the $n$-tuple becomes the zero vector $(0,\ldots,0)$. 
This implies that the complexity decreases after performing a finite sequence of IH-moves.
\end{proof}

Claim \ref{claim:reducing complexity} 
completes the induction step and hence the proof of Theorem \ref{thm:IH}.

\section{Poly-continuous patterns and networks}
\label{sec:tricontinuous}

One background of this study is constructing a mathematical model of structures made by diblock or triblock copolymers.

Diblock copolymers produces spherical, cylindrical, lamellar and bicontinuous structures. See \cite{Matsen2012} for example. Typical examples of bicontinuous structures are Gyroid, D-surface and P-surface\cite{Squires2005}. Mathematical model of such structures are triply periodic non-compact surfaces $F$ embedded in $\mathbb R^3$
which divide $\mathbb R^3$ into two possibly disconnected submanifolds $V_1$ and $V_2$ such that $V_1\cap V_2=\partial V_1\cap \partial V_2=F$. 
A subset of $\mathbb R^3$ is called {\em triply periodic} if it is invariant by the standard $\mathbb Z^3$ action on $\mathbb R^3$.
We call such a surface a {\em bicontinuous pattern} \cite{Hyde2000}.

We will consider the case where $V_1$ and $V_2$ are open neighborhood of networks.
Here a {\em network} means an infinite graph embedded in $\mathbb R^3$.
See, for example, \cite{Hyde2000,Hyde2009}.
In this case the bicontinuous pattern is uniquely determined by networks up to isotopy.
We say such a bicontinuous pattern is {\em associated to} a network.

On the other hand, for example triblock-arm star-shaped molecules yields a tricontinuous structure \cite{Campo2017}.
One mathematical model of such tricontinuous (resp. poly-continuous) structures is a triply periodic non-compact multibranched surface (or more generally polyhedron) dividing $\mathbb R^3$ into 3 (resp. several) possibly disconnected non-compact submanifolds $V_1$, $V_2$ and $V_3$ (resp. $V_1, \cdots, V_k$). 
We assume that each $V_i$ is the open neighborhood of three (resp. several) networks in $\mathbb R^3$.
We call such a multibranched surface a {\em tricontinuous pattern} (resp. {\em poly-continuous pattern})\cite{Hyde2000,Hyde2009,Hyde2012}.

The relation between poly-continuous patterns and networks is not obvious in this case.
Two different poly-continuous patterns are associated to one network and vice versa.
Here we will give a necessary and sufficient condition for poly-cotinuous patterns to give the same network.

By considering the quotient space of the standard $\mathbb Z^3$ action, we obtain a graph in the 3-dimensional torus $T^3$ as the quotient space of the triply periodic network and a compact multibranched surface in $T^3$ as the quotient space of the triply periodic poly-continuous pattern.

{\em IX-moves and XI-moves for triply periodic poly-continuous patterns in $\mathbb R^3$} are lifts of IX-moves and XI-moves of the corresponding multibranched surfaces in $T^3$,
By applying Theorem \ref{thm:IXXI}, two multibranched surfaces in $T^3$ corresponding to two poly-continuous pattern of a given triply periodic network can be related by a finite sequence of IX-moves, XI-moves and isotopies.

\begin{thm}\label{thm:poly}
Let $X$ and $X'$ be triply periodic poly-continuous patterns in $\mathbb R^3$ associated to a triply periodic network of multiple components.
Suppose that $X$ and $X'$ have no disk region.
Then $X$ can be transformed into $X'$ by a finite sequence of $IX$-moves, $XI$-moves and isotopies.
\end{thm}

By applying \cite{MS95}, two networks corresponding to one triply periodic poly-continuous pattern can be related by a finite sequence of lifts of edge-contractions and vertex-expansions of quotient graphs in $T^3$.

Study of tricontinuous patterns using decomposition of $T^3$ with a multibranched surface will be discussed in the forthcoming papers \cite{inpreparation}.

\section{A short remark on minors of multibranched surfaces}

As an analogy of graph minor (e.g. \cite{D}), Matsuzaki and the third author introduced the notion of 
{\it minor} of multibranched surfaces and studied intrinsic properties of  multibranched surfaces (\cite{MO17}).
However, in that paper, the authors took into consideration only IX-moves along normal annulus regions. 
Based on Theorem \ref{thm:IXXI}, it is more natural to define the minor of multibranched surfaces as follows.
In this section, we allow the degree $d_i$ of a branch locus $B_i$ to be $1$ or $2$ as well, and we assume that a multibranched surface is {\em regular} (i.e. for each branch locus $B_i$, the wrapping number of all components of $f^{-1}(B_i)$ is a divisor of the degree $d_i$ of $B_i$.

We consider regular multibranched surfaces modulo homeomorphism. 
Let $X$ and $Y$ be regular multibranched surfaces. 
We write $X\overset{\mathrm{r}}{<} Y$ if $X$  is obtained by removing a region of $Y$. 
We write $X\overset{\mathrm{c}}{<} Y$ if $X$ is obtained by contracting an normal annulus region, a quasi-normal annulus region or an normal M\"{o}bius-band region of $Y$.  
If $X\overset{\mathrm{r}}{<} Y$ or $X\overset{\mathrm{c}}{<} Y$, we write $X < Y$.

We denote by $\mathcal{M}$ the set of all regular multibranched surfaces (modulo homeomorphism).
We define an equivalence relation $\sim$ on $\mathcal{M}$ as follows: if $X < Y$ and $Y < X$, then $X \sim Y$. An element of the quotient set $\mathcal{M}/\sim$ is called a {\it multibranched surface class} (or a multibranched surface for simplicity).
We define a partial order $\prec$ on $\mathcal{M}/\sim$ as follows. 
Let $X$, $Y \in \mathcal{M}$. 
We denote $[X] \prec [Y]$ if there exists a finite sequence $X_0, X_1, \ldots, X_{n-1}, X_n $ of multibranched surfaces such that $X_0 \sim X$, $X_n \sim Y$ and $X_0 < X_1 < \cdots < X_{n-1} < X_n$.

A multibranched surface (class) $[X]$ is called a {\it minor} of a multibranched surface (class) $[Y]$ if $[X] \prec [Y]$. 
In particular, $[X]$ is called a {\it proper minor} of $[Y]$ if $[X] \prec [Y]$ and $[Y] \not= [X]$.
A subset $\mathcal{P}$ of $\mathcal{M}/\sim$ is said to be 
{\it minor closed} if for every multibranched surface {$[X] \in \mathcal{P}$, every minor of $[X]$ belongs to $\mathcal{P}$}.
For a minor closed set $\mathcal{P}$, we define the {\it obstruction set} $\Omega(\mathcal{P})$ as follows: 
\[ \Omega(\mathcal{P})= \left\{ [X] \in \mathcal{M}/\sim \mid 
[X] \not \in \mathcal{P}, \mbox{ Every proper minor of } [X] \mbox{ belongs to } \mathcal{P}. \right\} \]

With respect to the above notion, the results stated in \cite{MO17} still hold.
For example, the set of multibranched surfaces embeddable into $S^3$, denoted by $\mathcal{P} {S^3}$, is minor closed (\cite[Proposition 5.7]{MO17}), and for a multibranched surface $X$ in Fig. 11 of \cite{MO17} or equivalently $X=X_g(p_1,p_2,\ldots,p_n)$ in \cite{EO18}, if $p={\rm gcd} \{p_1, p_2, \ldots, p_n \}$ is not $1$, then $X \in \Omega(\mathcal{P} S^3)$.
As the philosophy of graph minor theory, the obstruction set for the embeddability into a 3-manifold 
reflects the properties of the 3-manifold.  
Finally, we propose the next problem which can be regarded as a 2-dimensional version of Kuratowski's and Wagner's theorems.

\begin{prob}
Characterize the obstruction set $\Omega(\mathcal{P} S^3)$.
\end{prob}

It is known that the following multibranched surfaces belong to $\Omega(\mathcal{P} S^3)$.

\begin{itemize}
\item non-orientable closed surfaces.
\item $X_1$ in \cite[Theorem 3.2]{EMO}, where $|ad_{e_1}(l_1)|\ge 2$.
\item $X_2$ in \cite[Theorem 3.3]{EMO}.
\item $X_3$ in \cite[Theorem 3.7]{EMO}.

\item $X$ in \cite[Example 5.8]{MO17} or \cite{EO18}, where $p={\rm gcd} \{p_1, p_2, \ldots, p_n \}\ne 1$.
\item $X$ in \cite[Example 5.9]{MO17}.

\end{itemize}

\end{document}